\documentclass[11pt]{article}
\usepackage{graphicx}
\usepackage{amsthm, amsmath, amssymb, tikz, bbm}
\usepackage{mathtools}
\usepackage{xifthen}
\usepackage{comment}
\usepackage[colorlinks=true,linkcolor=blue,citecolor=blue, urlcolor=blue]{hyperref}

\usepackage[margin=1in]{geometry} 

\usepackage[shortlabels]{enumitem}
\usepackage{todonotes}
\usetikzlibrary{calc,shapes, backgrounds}

\newtheorem{lemma}{Lemma}
\newtheorem{theorem}{Theorem}
\newtheorem{corollary}{Corollary}

\newtheorem{claim}{Claim}

\newcommand{\dss}{\displaystyle\sum}

\newcommand{\lp}{\left (}
\newcommand{\rp}{\right )}

\newcommand{\cF}{\mathcal{F}}

\newcommand{\cG}{\mathcal{G}}

\newcommand{\I}{\mathcal{I}}

\newcommand{\C}{\mathcal{C}}

\newcommand{\CR}{{\rm cr}}

\newcommand{\commentout}[1]{}

\title{Anti-Ramsey number of disjoint rainbow bases in all matroids}
\author{
Linyuan Lu
\thanks{University of South Carolina, Columbia, SC 29208,
({\tt lu@math.sc.edu}). This author is partially supported by NSF grant 2038080.}
\and
Andrew Meier \thanks{University of South Carolina, Columbia, SC 29208,
({\tt am66@mailbox.sc.edu}).} 
}

\begin{document}

\maketitle
\begin{abstract}
Consider a matroid $M=(E,\I)$ with its elements of the ground set $E$ colored.  
A {\em rainbow basis} is a maximum independent set in which each element 
 receives a different color.  The {\em rank} of a subset $S$ of $E$, denoted by $r_M(S)$, is the maximum size of an independent set in $S$. A {\em flat} $F$ is a maximal set in $M$ with a fixed rank.
 The {\em anti-Ramsey} number of $t$ pairwise disjoint rainbow bases in $M$, denoted by $ar(M,t)$, is defined as the maximum number of colors $m$ such that there exists an $m$ coloring of the ground set $E$ of $M$ which contains no $t$ pairwise disjoint rainbow bases.  We determine $ar(M,t)$ for all matroids of rank at least 2: $ar(M,t)=|E|$ if there exists a flat $F_0$ with $|E|-|F_0|<t(r_M(E)-r_M(F_0))$; and 
$ar(M,t)=\max_{F\colon r_M(F)\leq r_M(E)-2} \{|F|+t(r_M(E)-r_M(F)-1)\}$ otherwise. 
This generalizes Lu-Meier-Wang's previous result on the anti-Ramsey number of
edge-disjoint rainbow spanning trees in any multigraph $G$.
\end{abstract}

\section{Introduction}
Anti-Ramsey problems have been investigated for various types of structures. For example, take a multigraph $G$ which is a graph that permits multiple edges between two vertices. Let $V(G)$ and $E(G)$ denote the vertex set and edge set of $G$ respectively. An edge-colored graph $G$ is called \textit{rainbow} if every edge of $G$ receives a different color.
One type of anti-Ramsey problem asks for the maximum number of colors $AR(K_n,\cG)$ in an edge-coloring of $K_n$ containing no rainbow copy of any graph in a class $\cG$. For some earlier results when $\cG$ consists of a single graph, see the survey \cite{FMO10}. In particular, Montellano-Baallesteros and Neumann-Lara \cite{Montella-Nemann05} showed a conjecture of Erd\H{o}s, Simonovits and S\'os \cite{ESS75} by computing $AR(K_n, C_k)$. Jiang and West \cite{Jiang-West04} determined the anti-Ramsey number of the family of trees with $m$ edges. For some more recent results, see for example \cite{FGLX2021, Gorgol2016, GLS2020, Jiang-Pikhurko2009, LSS2019, Lu-Meier-Wang-AntiRamsey, Xie-Yuan2020, Yuan2021+, Yuan-Zhang2021+}.

In particular, anti-Ramsey problems have been investigated for rainbow spanning subgraphs. For example, Hass and Young \cite{Haas-Young12} showed that the anti-Ramsey number for perfect matchings (when $n$ is even) is $\binom{n-3}{2}+2$ for $n\geq 14$. For rainbow spanning trees, let $ar(K_n,t)$ be the maximum number of colors in an edge-coloring of $K_n$ not having $t$ edge-disjoint rainbow spanning trees.  (In the literature, the notation $r(K_n,t)$ was often used for anti-Ramsey number. Here we change it to $ar(K_n,t)$ to avoid the confusion with the rank function $r(\cdot)$.)
Bialostocki and Voxman \cite{Bialostocki-Voxman01} showed that $ar(K_n,1) = \binom{n-2}{2}+1$ for $n \geq 4$. Akbari and Alipour \cite{Akbari-Alipour07} showed that $ar(K_n,2) = \binom{n-2}{2}+2$ for $n\geq 6$. Jahanbekam and West \cite{Jahanbekam-West16} extended the investigations to arbitrary number of edge-disjoint rainbow spanning trees (along with the anti-Ramsey number of some other edge-disjoint rainbow spanning structures such as matchings and cycles). In particular, for $t$ edge-disjoint rainbow spanning trees, they showed that
	 $$ar(K_n,t) = \begin{cases} 
	 \binom{n}{2}-t & \textrm{ for } n = 2t, \\
    \binom{n-2}{2} +t   & \textrm{ for } n >  2t+\sqrt{6t-\frac{23}{4}}+\frac{5}{2},
 		      \end{cases}$$
and they \cite{Jahanbekam-West16} conjectured that $ar(K_n,t) = \binom{n-2}{2} + t$ whenever $n\geq 2t+2 \geq 6$. This conjecture was recently settled by the first author and Wang in \cite{Lu-Wang-AntiRamsey}.
Together with previous results \cite{Bialostocki-Voxman01, Akbari-Alipour07, Jahanbekam-West16}, they give the anti-Ramsey number of $t$ edge-disjoint rainbow spanning trees in $K_n$ for all values of $n$ and $t$. 
\begin{theorem}[Lu-Wang \cite{Lu-Wang-AntiRamsey}]\label{anti-Ramsey}
	For all positive integers $n$ and $t$,
 $$ar(K_n,t) = \begin{cases} 
\binom{n}{2} &   \textrm{ for } n<2t,\\
 t(n-2) & \textrm{ for }2t\leq  n \leq 2t+1,\\
 \binom{n-2}{2}+t   & \textrm{ for } n \geq 2t+2.
 \end{cases}$$
\end{theorem}

The authors and Wang \cite{Lu-Meier-Wang-AntiRamsey} showed a general framework that enabled one to determine the anti-Ramsey number $r(G,t)$  for any host multigraph $G$ and all values of $t$. Note that we assume that $G$ is loopless, since otherwise we can always assign each loop a unique distinct color without affecting the existence of rainbow spanning trees. 

Given a graph $H$ and a partition $P$ of the vertex set of $H$, let $\CR(P,H)$ be the set of crossing edges in $H$ whose end vertices belong to different parts in the partition $P$. Let $E(P,H)$ be the set of non-crossing edges of $H$ with respect to $P$, i.e., edges whose two end vertices are contained in the same part in the partition $P$. Let $|P|$ denote the number of parts in $P$. It was shown that

\begin{theorem}[Lu-Meier-Wang \cite{Lu-Meier-Wang-AntiRamsey}]\label{mainST}
For any multigraph $G$, if there is a partition $P_0$ of vertices of $G$ satisfying $|E(G)|-|E(P_0,G)|<t(|P_0|-1)$, then $r(G,t)=|E(G)|$. Otherwise, 
\begin{equation}\label{eq:main1}
ar(G,t)=\max_{P\colon |P| \geq 3} \{|E(P,G)|+t(|P|-2)\},
\end{equation}
where the maximum is taken among all partitions $P$
(with $|P|\geq 3$) of the vertex set of $G$.
\end{theorem}

The fundamental structure underpinning the ideas involved in many anti-Ramsey type problems, especially those of spanning substructures, is that of a matroid. In this paper, we generalize the main result of Lu-Meier-Wang in \cite{Lu-Meier-Wang-AntiRamsey} to the framework of matroids. Given a matroid $M=(E,\I)$, where $E$ is the (colored) ground set and $\I$ is the family of independent sets.  A {\em rainbow basis} is a maximum independent set in which each element  receives a different color.  The {\em rank} of a set $S\subseteq E$, written by $r_M(S)$,  is the maximum size of an independent set in $S$. A {\em flat} $F$ is a maximal set in $M$ with a fixed rank. The \textit{anti-Ramsey} number of $t$ pairwise disjoint rainbow bases in $M$, denoted by $ar(M,t)$, is defined as the maximum number of colors $m$ such that there exists an $m$ coloring of the ground set $E$ of $M$ which contains no $t$ pairwise disjoint rainbow bases. Please see the detailed definitions of these terminologies in Section \ref{Notation}. Our main result is

\begin{theorem}\label{main}
For any matroid $M=(E,\I)$ with rank at least 2, if there is a flat $F$ of $M$ satisfying $|E|-|F|<t(r_M(E)-r_M(F))$, then $ar(M,t)=|E|$. Otherwise, 
\begin{equation}\label{eq:main}
ar(M,t)=\max_{F\colon r_M(F) \leq r_M(E)- 2} \{|F|+t(r_M(E)-r_M(F)-1)\},
\end{equation}
where the maximum is taken among all flats $F$ of $M$
with $r_M(F)\leq r_M(E)-2$.
\end{theorem}

We also include the elementary result in the case where the rank of $M$ is 1. Given a matroid $M$ define $M_0$ as all rank 0 elements.

\begin{theorem}\label{rank1}
For any matroid $M=(E,\I)$ of rank 1,
we have
$$
ar(M,t)=\begin{cases}
|E| & \mbox{ if } |E|<|M_0|+t,\\
|M_0| +t -1 &\mbox{ otherwise.}
\end{cases}
$$
\end{theorem}

\noindent
The above result is a straightforward application of the pigeonhole principle. 

When $r_M(E)=2$, $r_M(F)=0$ implies $F\subseteq M_0$, while $r_M(F)=1$ implies $F\subseteq cl(x)$, the closure of some rank 1 element $x$. (See the definitions in Section \ref{Notation}.)
Theorem \ref{main} implies the following corollary.
\begin{corollary}\label{rank2}
For any matroid $M=(E,\I)$ of rank 2, 
we have
$$
ar(M,t)=\begin{cases}
|E| & \mbox{ if } |E|<|M_0|+2t \mbox{ or } |E|< |cl(x)|+t \mbox{ for some } x.\\
|M_0| +t &\mbox{ otherwise.}
\end{cases}
$$
\end{corollary}

The paper is organized as follows. In Section \ref{Notation}, we provide the machinery needed to prove Theorem \ref{main}. We turn to the proof of Theorem \ref{main} in Section \ref{MainProof}, and in Section \ref{applications} we show Theorem \ref{mainST} is a direct consequence of Theorem \ref{main} and also deduce the anti-Ramsey number of some nice matroids.

\section{Notation and tools}\label{Notation}

A matroid $M=(E,\I)$ is a set $E$ called a ground set and a collection of subsets $\I\subseteq 2^{E}$ called independent sets satisfying, 
\begin{enumerate}
    \item[(I1)] if $A\in \I$ and $A'\subseteq A$ then $A'\in\I$, and
    \item[(I2)] if $A\in \I$ and $B\in \I$ with $|A|<|B|$ then there exists a $b\in B\setminus A$ such that $A\cup \{b\}\in \I$. 
\end{enumerate}

Any matroid discussed in this paper will be assumed to have a finite ground set. Any set which is not independent is \emph{dependent}. A minimum dependent set is called a \emph{circuit}, which satisfies
\begin{enumerate}
    \item [(C1)] No circuit is contained in another circuit.
    \item [(C2)] For any two distinct circuits $C_1$ and $C_2$ and $x\in C_1 \cap C_2$, then $C_1\cup C_2\setminus \{x\}$ contains a circuit.
\end{enumerate}

An independent set with maximal size is called a \emph{basis}, which satisfies
\begin{enumerate}
    \item[(B1)] No basis is contained in any other basis.  
    \item[(B2)] If $B_1$ and $B_2$ are distinct bases, then for every $x\in B_1\setminus B_2$, there is $y\in B_2\setminus B_1$ such that
     $B_1\setminus \{x\} \cup \{y\}$ is a basis. 
\end{enumerate}

From (B1) and (B2) one can easily derive that any two bases of a matroid have the same size. This number is called the \emph{rank} of $M$. The \emph{rank function} $r_M:2^{E}\rightarrow \mathbb{N}$ is defined as 
$$r_M(S)=\max\{|I|:I\in\I \text{ and } I\subseteq S\}.$$
When $M$ is clear under context, we may omit the subscript and write
the rank function as $r(\cdot)$.
The rank function $r_M(\cdot)$ satisfies
\begin{enumerate}
    \item[(R1)] For any set S, $0\leq r_M(S)\leq |S|$. 
    \item[(R2)] If $S\subseteq T$, then $r_M(S)\leq r_M(T) $.
    \item[(R3)] If $S,T\subseteq E$, then the following {\rm semimodular inequality} holds:
    \begin{equation}
        \label{rank_submodular}
        r_M(S)+r_M(T)\geq r_M(S\cup T) +r_M(S\cap T).
    \end{equation}
\end{enumerate}

The {\em closure} of $S$ is the set $cl(S)=\{x\colon r(S\cup \{x\})=r(S)\}$. It satisfies the following properties
\begin{enumerate}
    \item [(S1)] For any set $S$, $S\subseteq cl(S)$.
    \item [(S2)] For any set $S$, $cl(cl(S))=cl(S)$.
    \item [(S3)] If $S\subset T$, then $cl(S)\subseteq cl(T)$.
    \item [(S4)] If $x\not \in cl(S)$ but $x\in cl(S\cup \{y\})$, then $y\in cl(S\cup \{x\}).$
\end{enumerate}
When $S=\{x\}$, we write $cl(x)$ instead of $cl(\{x\})$, for short. A set $F$ is called a {\em flat} of $M$ if $cl(F)=F$.  The \emph{flats} of $M$ are the maximal sets for its rank.  

We will extensively use these properties in the paper. It is known that these properties of circuits, bases, rank, and closure can be used as alternative definitions of a matroid (see Whitney \cite{Whitney1935}, Wilson \cite{Wilson1973}, and Oxley\cite{Oxley1992}.)

Given $t$ matroids, $\{M_i=(E_i, \I_i)\colon i=1,2,\ldots t\}$,  the \emph{union matroid}, written by $M_1\vee M_2\vee \cdots M_t$, has a ground set
$\cup_{i=1}^t E_i$ and independent sets $\I=\{I_1\cup I_2\cup \cdots \cup I_t : I_i\in \I \text{ for all } i \in [t]\}.$ 
The celebrated matroid union theorem (see Nash-Williams \cite{Nash-Williams67}, Edmonds \cite{Edmonds70}) states
$$r_{\vee_{i=1}^t M_i}(S)=\min_{T\subseteq S}\{ |S\setminus T|+ \sum_{i=1}^t r_{M_i}(T\cap E_i)\}.$$
Let $M^t=M\vee M\vee\cdots \vee M$ be the $t$\emph{-fold union} of $M$. 
The above can be reformulated as
$$r_{M^t}(E)=\min_{F\subseteq E, F \text{ is a flat}}\{|E\setminus F|+t\cdot r_M(F)\}.$$

Given two matroids $M_1=(E,\I_1)$ and $M_2=(E, I_2)$ on the same ground set $E$, 
the matroid intersection theorem states that 
$$\max_{I\in \I_1\cap I_2}|I|=\min_{U\subseteq E}\{r_{M_1}(U)+r_{M_2}(E\setminus U)\}.$$
The matroid intersection theorem can be derived from the matroid union theorem, and vice versa (see Edmonds \cite{Edmonds70}).

The following result is due to Edmonds (see \cite{Edmonds65}), which gives a useful existence condition for $t$ pairwise disjoint bases in a matroid:

\begin{theorem}[Edmonds \cite{Edmonds65}] \label{existence}
For any matroid $M=(E,\I)$, $M$ contains $t$ pairwise disjoint bases if and only if every flat $F$ of $M$ satisfies $|E|-|F|\geq t(r_M(E)-r_M(F))$.
\end{theorem}

\begin{corollary}\label{degeneratecase}
Let $M=(E,\I)$ be a matroid. If there exists a flat $F$ of $M$ such that $|E|-|F|< t(r_M(E)-r_M(F))$, then $ar(M,t)=|E|$.
\end{corollary}

\begin{proof}
By Theorem \ref{existence}, $M$ cannot contain $t$ pairwise disjoint bases, and thus no coloring of the ground set can contain $t$ pairwise disjoint rainbow bases. As such, we can color every element of the ground set a unique color and so $ar(M,t)=|E|$. 
\end{proof}

Given a matroid $M=(E,\I)$ and a coloring of its ground set, we say a set of bases is \emph{color-disjoint} if no color among the bases is used more than once. Given a subset $S\subseteq E$, we use $c(S)$ to denote the set of colors which appear in $S$. When $S=\{x\}$ for some element $x\in E$, we write $c(x)$ as a shorthand. We will need the following existence result for $t$ color-disjoint bases due to Schrijver
(see \cite{Schrijver03} and \cite{Lu-Wang-AntiRamsey}).  


\begin{theorem} 
\label{colordisjointbases}
Let $M=(E,\I)$ be a matroid equipped with a coloring of its ground set $E$. $M$ contains $t$ color-disjoint bases if and only if every flat $F$ of $M$ satisfies 
$$|c(E\setminus F)|\geq t(r_M(E)-r_M(F)).$$
\end{theorem}

From this point on, we refer to the condition that every flat $F$ of $M$ satisfies $|E|-|F|\geq t(r_M(E)-r_M(F))$ as the $t$ pairwise disjoint bases (tPDB) condition. Furthermore, where there is no room for confusion, we will not specify that $F$ is a flat. 
Given a family of $t$ pairwise disjoint rainbow independent sets $\cF=\{I_1,I_2,\cdots I_t\}$ we say that $\cF$ has a color-disjoint extension to $t$ pairwise disjoint rainbow bases if elements can be added to each $I_i$ so that each $I_i$ becomes a basis $\mathbf{I}_i$ and if $c$ is a color in $\mathbf{I}_i\setminus I$ then $c$ does not appear in $I_i$ and any other $\mathbf{I}_j$, $i\neq j$. The following lemma generalizes Theorem \ref{colordisjointbases} (since we can set all $I_i$'s to be the empty set).

\begin{lemma}\label{colordisjointextension}
Let $M=(E,\I)$ be a matroid with a coloring of $E$ and let $\cF=\{I_1,I_2,\cdots I_t\}$ be a family of $t$ pairwise disjoint rainbow independent sets. $\cF$ has a color disjoint extension to $t$ pairwise disjoint rainbow bases in $M$ if and only if for every flat $F$ of $M$
\begin{equation}\label{eqn:extension}
|c(E\setminus F)\setminus c(\cup_{i=1}^t I_i)|+\sum\limits_{i=1}^t|I_i\cap (E\setminus F)|\geq t(r_M(E)-r_M(F)).
\end{equation}
\end{lemma}

\begin{proof}
For $1\leq i \leq t$, define $i$-th auxiliary matroid $M_i'=(E', \I_i')$ where $E'$ is obtained from $E$ by removing all elements whose colors appear in $\cF$ and 
$$\I'_i=\{S\colon S\subseteq E' \mbox{ and } S\cup I_i\in \I\}.$$
The subset $E'$ inherits the coloring from $E$.  

Let $M'=M_1'\vee M_2'\vee\cdots\vee M_t'$ be the union matroid, and let $M''=(E',\I'')$ be the matroid with ground set $E'$ and whose independent sets are the rainbow subsets of $E'$. By the matroid union theorem,
$$r_{M'}(S)=\min_{T\subseteq S }\{|S\setminus T|+\sum_{i=1}^t r_{M_i'}(T)\}.$$
By the matroid intersection theorem,
$$\max_{I\in \I'\cap \I''}|I|=\min_{U\subseteq E'}\{r_{M'}(U)+r_{M''}(E'\setminus U)\}.$$ 
Since each $I_i$ can be extended in a color-disjoint way, it follows that 
\begin{eqnarray*}
\sum_{i=1}^t r_{M_i'}(E')&=&\max_{I\in \I'\cap \I''}|I|\\
&=&\min_{U\subseteq E'}\{r_{M'}(U)+r_{M''}(E'\setminus U)\}\\
&=&\min_{U\subseteq E'}\{\min_{T\subseteq U}\{|U\setminus T|+\sum_{i=1}^r r_{M'_i}(T)\}+r_{M''}(E'\setminus U)\}.
\end{eqnarray*}
Let $T\subseteq U\subseteq E'$ be the $T$ and $U$ which achieve the minimum above such that $U$ has the minimum number of elements. Notice that moving any element from $U\setminus T$ into $E'\setminus U$ either decreases the minimum above (contradicting the choice of $T$ and $U$) or it leaves the above minimum the same (thus contradicting the minimality of $U$). As such, we can assume $T=U$. It follows that 
$$\sum_{i=1}^t r_{M'_i}(E')=\min_{F\subseteq E': F \text{ is a flat }}\left\{\sum_{i=1}^t r_{M'_i}(F)+r_{M''}(E'\setminus F)\right\}.$$
If $F'$ is any flat of $M$, it follows that 
\begin{eqnarray*}
\sum_{i=1}^r r_{M'_i}(E')&=&\min_{F\subseteq E': F \text{ is a flat }}\left\{\sum_{i=1}^t r_{M'_i}(F)+r_{M''}(E'\setminus F)\right\}\\
&\leq & \sum_{i=1}^t r_{M'_i}(F')+r_{M''}(E'\setminus F')\\
&=&\sum_{i=1}^t r_{M'_i}(F')+|c(E\setminus F')\setminus c(\cup_{i=1}^tI_i)|.
\end{eqnarray*}
The previous line implies 
\begin{equation}\label{forwardextension}
|c(E\setminus F')\setminus c(\cup_{i=1}^tI_i)|\geq \sum_{i=1}^t(r_{M'_i}(E')-r_{M'_i}(F')).
\end{equation}
Notice that
\begin{equation}\label{rankE'}
r_{M_i'}(E')= r_M(E)-|I_i|
\end{equation}
and
\begin{equation}\label{rankF'}
r_{M_i'}(F')=r_M(F')-|I_i\cap F'|.
\end{equation}
These facts combined with Equation \ref{forwardextension} imply
$$|c(E\setminus F')\setminus c(\cup_{i=1}^t I_i)|+\sum\limits_{i=1}^t|I_i\cap (E\setminus F')|\geq t(r_M(E)-r_M(F')).$$
Since $F'$ was chosen arbitrarily as a flat of $M$, we are finished.

Now we turn to the reverse direction. Suppose we have a coloring of $M$ such that for any flat $F$ of $M$,
$$|c(E\setminus F)\setminus c(\cup_{i=1}^t I_i)|+\sum\limits_{i=1}^t|I_i\cap (E\setminus F)|\geq t(r_M(E)-r_M(F)).$$
In a similar fashion to the forward direction, we obtain

\begin{eqnarray*}
\max_{I\in \I'\cap \I''}|I|&=&\min_{F\subseteq E' : F \text{ is a flat }}\left\{\sum_{i=1}^t r_{M'_i}(F)+r_{M''}(E'\setminus F)\right\}\\
&=&\sum_{i=1}^t r_{M'_i}(F_0)+|c(E'\setminus F_0)| \hspace*{2cm} \mbox{assuming the minimum is acheived at $F_0$}
\\
&\geq & \sum_{i=1}^t r_{M'_i}(F_0)+t(r_M(E)-r_M(F_0))-\sum\limits_{i=1}^t|I_i\cap (E\setminus F_0)| \\
&=&\sum_{i=1}^t r_{M'_i}(E').
\end{eqnarray*}
The last equality holds due to Equations \ref{rankE'} and \ref{rankF'}. It follows that $\cF$ has a color-disjoint extension to $t$ pairwise disjoint rainbow bases in $M$.
\end{proof}

\begin{lemma}\label{lowerbound}
If $M=(E,\I)$ is a matroid with rank at least 2 satisfying the (tPDB) condition, then 
$$ar(M,t)\geq\max_{F\colon r_M(F) \leq r_M(F)- 2} \{|F|+t(r_M(E)-r_M(F)-1)\},$$
where the maximum is taken among all flats $F$ (with $r_M(F)\leq r_M(E)-2$ of $M$).
\end{lemma}

\begin{proof}
We exhibit a coloring of the ground set of $M$ which contains no $t$ pairwise disjoint rainbow bases. We first consider the case when $t=1$. When $t=1$, we are in the color-disjoint setting so Theorem \ref{colordisjointbases} applies. Let $F$ be the flat of $M$ which achieves the maximum above. Color $F$ in unique colors and color the remaining $|E\setminus F|$ in $r_M(E)-r_M(F)-1$ colors. Since $|c(E\setminus F)|= r_M(E)-r_M(F)-1<r_M(E)-r_M(F)$, $F$ violates Theorem \ref{colordisjointbases} and $M$ contains no rainbow basis. 

We now assume $t\geq 2$. Let $F$ be any flat of $M$ with $r_M(F)\leq r_M(E)-2$. Color every element of $F$ a unique color. This contributes $|F|$ colors. Since $M$ satisfies the (tPDB) condition and $r_M(E)-r_M(F)\geq 2$, we can color $t(r_{M}(E)-r_M(F)-1)-1$ elements of $E\setminus F$ unique colors, and color the remaining elements of $E\setminus F$ in a new color $c_s$. This contributes an additional $t(r_{M}(E)-r_M(T)-1)$ colors. It remains to show that $M$ contains no $t$ pairwise disjoint rainbow bases. 

Suppose $M$ contains $t$ pairwise disjoint rainbow bases $\{B_1,B_2,\cdots, B_t\}$. By restricting $M$ to the set $E\setminus F$, it is clear that each $B_i$ contributes $|B_i|-|B_i\cap F|-1$ elements not in color $c_s$ to $E\setminus F$. Since $B_i\cap F$ is independent in $F$, $|B_i\cap F|\leq r_M(F)$. It follows that
\begin{eqnarray*}
|\{e\in E\setminus F : c(e)\neq c_s\}|&\geq & \sum_{i=1}^t(|B_i|-|B_i\cap F|-1)\\
&\geq & t(r_M(E)-r_M(F)-1).
\end{eqnarray*}
But, $E\setminus F$ has only $t(r_{M}(E)-r_M(F)-1)-1$ elements which are not in color $c_s$. This is a contradiction and thus 
$$ar(M,t)\geq\max_{F\colon r_M(F) \leq r_M(E)- 2} \{|F|+t(r_M(E)-r_M(F)-1)\}.$$
\end{proof}

Now we assume $M$ is a matroid with rank at least two and is equipped with a coloring of $E$ with $\max_{F\colon r_M(F) \leq r_M(E)- 2} \{|F|+t(r_M(E)-r_M(F)-1)\}+1$ colors.
For a flat $F$ of $M$, we define
\begin{align*}
\eta(F)&:=|F|-|c(F)|\\
\shortintertext{and}
\xi(F)&:=|c(F)\cap c(E\setminus F)|.
\end{align*}
A quick application of the inclusion-exclusion principle shows
\begin{eqnarray*}
|c(E)|-|F|&=&|c(F)\cup c(E\setminus F)|-|F|\\
&=&|c(F)|+|c(E\setminus F)|-|c(F)\cap c(E\setminus F)|-|F|\\
&=&|c(E\setminus F)|-\eta(F)-\xi(F).
\end{eqnarray*}
It follows that 
\begin{equation}\label{eq:E/TColors}
 |c(E\setminus F)| = |c(E)| -|F| +\eta(F) +\xi(F). 
\end{equation}
Since $c(E) \geq \max_{F\colon r_M(F) \leq r_M(E)- 2} \{|F|+t(r_M(E)-r_M(F)-1)\}+1$, we then have the following corollary:
\begin{corollary}\label{interiorcolors} 
Let $M=(E,\I)$ be a matroid with a coloring of $E$ with at least $\max_{F\colon r_M(F) \leq r_M(E)- 2} \{|F|+t(r_M(E)-r_M(F)-1)\}+1$ colors appearing in $E$. Then for any flat $F$ of $M$, we have
\begin{enumerate}
    \item if $r_M(F) \leq r_M(E)- 2$, then
    \begin{equation}\label{eq:T2-cross}
    |c(E\setminus F)|\geq t(r_M(E)-r_M(F)-1)+\eta(F) + \xi(F) +1;
    \end{equation}
    \item if $r_M(F) = r_M(E)- 1$ and $|c(E)|>|F|$, then
       \begin{equation}\label{eq:T1-cross}
       |c(E\setminus F)|\geq \eta(F) + \xi(F) +1.
       \end{equation}
\end{enumerate}
\end{corollary}

An important observation is that the flats which satisfy Corollary \ref{interiorcolors} are not ``far" from satisfying Inequality \ref{eqn:extension}. This observation is the intuition behind the following lemma:

\begin{lemma}\label{subindependentlargeenough} Let $\I = \{I_1,I_2,\ldots,I_t\}$ be a set of $t$ pairwise disjoint rainbow independent sets such that $\sum\limits_{i=1}^t|I_i|\geq t-1+|c(\cup_{i=1}^t I_i)|$, and each color appearing in $\cup_{i=1}^t I_i$ appears at least twice. Then for any flat satisfying Inequality \eqref{eq:T2-cross} (when $r_M(F) \leq r_M(E)- 2$) and Inequality \eqref{eq:T1-cross} (when $r_M(F) = r_M(E)- 1$), we have that
\begin{equation}\label{extensioninequality}|c(E\setminus F)\setminus c(\cup_{i=1}^t I_i)|+\sum\limits_{i=1}^t|I_i\cap (E\setminus F)|\geq t(r_M(E)-r_M(F)).
\end{equation}
\end{lemma}

\begin{proof}
Observe that given any flat $F$, 
$$\dss_{i=1}^t |I_i| = \dss_{i=1}^t |I_i\cap(E\setminus F)| + \dss_{i=1}^t |I_i\cap F|.$$
Since $\sum\limits_{i=1}^t|I_i|\geq t-1+|c(\cup_{i=1}^t I_i)|$, it follows that 
\begin{equation}
    \dss_{i=1}^t |I_i\cap(E\setminus F)| \geq (t-1) + |c(\displaystyle\cup_{i=1}^t I_i)| - \dss_{i=1}^t |I_i\cap F|. \label{eq:cl1a}
\end{equation}

Hence it follows that 
\begin{align*}
&\hspace*{-1cm} |c(E\setminus F)\setminus c(\cup_{i=1}^t I_i)|+\sum_{i=1}^t |I_i\cap (E\setminus F)|- t(r_M(E)-r_M(F)) \\
 &=  |c(E\setminus F)|- |c\left(\cup_{i=1}^t I_i\right)\cap c(E\setminus F)|+\sum_{i=1}^t|I_i\cap (E\setminus F)| - t(r_M(E)-r_M(F)) \\
 &\geq  \lp t(r_M(E)-r_M(F)-1)  + \eta(F) + \xi(F)+1\rp  - |c\left(\cup_{i=1}^t I_i\right)\cap c(E\setminus F)| \hspace*{1cm} \mbox{ by  Corollary \ref{interiorcolors}}  
 \\
   & \qquad+ \lp (t-1) + |c(\displaystyle\cup_{i=1}^t I_i)| - \dss_{i=1}^t |I_i\cap F| \rp - t(r_M(E)-r_M(F)) 
   \hspace*{1cm}
\mbox{ and Inequality \eqref{eq:cl1a}}   \\
 &=  |c(\displaystyle\cup_{i=1}^t I_i)| + \eta(F) + \xi(F) - \lp \dss_{i=1}^t |I_i\cap F|  + |c\left(\cup_{i=1}^t I_i\right)\cap c(E\setminus F)| \rp \\
 &\geq |c(\cup_{i=1}^t I_i)| + \eta(\cup_{i=1}^t I_i) + \xi(\cup_{i=1}^t I_i) - \lp \dss_{i=1}^t |I_i\cap F|  + |c\left(\cup_{i=1}^t I_i\right)\cap c(E\setminus F)| \rp \\
 &=  0,
\end{align*}
where the equality in the last step is essentially Equation \eqref{eq:E/TColors} by replacing $F$ with $\cup_{i=1}^t I_i$.
This completes the proof of the lemma.
\end{proof}

\section{Proof of Theorem \ref{main}}\label{MainProof}

\begin{proof}[Proof of Theorem \ref{main}]
We first handle the case of $t=1$. The lower bound is given by Lemma \ref{lowerbound}. Let $M=(E,\I)$ be a matroid and color $E$ with $\max_{F : r_M(F)\leq r_M(E)-2}\{|F|+r_M(E)-r_M(F)-1\}+1$ colors. When $t=1$, color-disjoint bases and disjoint rainbow bases are the same, thus Theorem \ref{colordisjointbases} applies. If $F$ is a flat with $r_M(F)\leq r_M(E)-2$, then 
\begin{eqnarray*}
|c(E\setminus F)|&\geq& |c(E)|-|F|\\
&\geq & |F|+r_M(E)-r_M(F)-|F|\\
&=& r_M(E)-r_M(E).
\end{eqnarray*}
As such, the condition of Theorem \ref{colordisjointbases} holds for every flat with $r_M(F)\leq r_M(E)-2$. If $r_M(F)=r_M(E)-1$, then the condition of Theorem \ref{colordisjointbases} reads $|c(E\setminus F)|\geq r_M(E)-r_M(F) =1$. Since $r_M(F)=r_M(E)-1$, it follows that $F\neq E$, and as such there is a color on $E\setminus F$. It follows that every flat satisfies Theorem \ref{colordisjointbases} and as such $M$ contains a rainbow basis. 

We now assume $t\geq 2$. The lower bound is given by Lemma \ref{lowerbound}. For the sake of contradiction, let $M=(E,\I)$ be a minimal (w.r.t. $|E|$) matroid with rank at least 2 that satisfies the (tPDB) condition which is a counter-example to the claim 
$$ar(M,t)\leq\max_{F\colon r_M(F) \leq r_M(E)- 2} \{|F|+t(r_M(E)-r_M(F)-1)\},$$
i.e., there exists a coloring of the elements of $E$ with $\max_{F\colon r_M(F) \leq r_M(E)- 2} \{|F|+t(r_M(E)-r_M(F)-1)\}+1$ colors with no $t$ pairwise disjoint rainbow bases. We can assume $M$ contains no rank-0 elements, since coloring all such elements in a unique color does not affect the existence of $t$ pairwise disjoint rainbow bases. Fix a coloring of the elements of $E$ with the described number of colors. Let $E_2$ be the subset of $E$ which consists of all elements whose color appears more than once in $E$. For any element $x\in E_2$ define
$$cl_2(x)=cl(x)\cap E_2=\{y\in E_2: r_{M}(\{x,y\})=1\}\cup\{x\}.$$
I.e. $cl_2(x)$ is the set of elements in $E_2$ which form 2-cycles with $x$ together with $x$.
We make the following claim:

\begin{claim}\label{2cycles}
For every $x\in E_2$, $|cl_2(x)|\leq t$.
\end{claim}

\begin{proof}
Suppose there exists an $x\in E_2$ such that $|cl_2(x)|\geq t+1$. Since $M$ contains $t$ pairwise disjoint bases, there must be at least one $y\in cl_2(x)$ such that $y$ is in none of these disjoint bases. Let $M\setminus y$ be the restriction matroid. Since $y$ is in none of the disjoint bases, $M\setminus y$ satisfies the (tPDB) condition and furthermore $r_{M\setminus y}(E\setminus y)=r_M(E)$. 
By the minimality of $M$ and the fact that the color on $y$ appears elsewhere in $M$,
\begin{eqnarray*}
ar(M\setminus y,t)&\leq& \max_{F\colon F\subseteq {E\setminus y},\; r_{M\setminus y}(F) \leq r_{M\setminus y}(E\setminus y)- 2} \{|F|+t(r_{M\setminus y}(E\setminus  y)-r_{M\setminus y}(F)-1)\}\\
&\leq & \max_{F\colon F\subseteq {E\setminus y}, \; r_{M}(F) \leq r_{M}(E)- 2} \{|F|+t(r_M(E)-r_M(F)-1)\}\\
&\leq& \max_{F\colon r_{M}(F) \leq r_{M}(E)- 2} \{|F|+t(r_M(E)-r_M(F)-1)\}\\
&<& |c(E)|\\
&=& |c(E\setminus y)|.
\end{eqnarray*}
As such, $M\setminus y$ and thus $M$ contains $t$ pairwise disjoint rainbow bases contradicting the choice of $M$. 

\end{proof}

\begin{claim}\label{BaseCase}
For every flat $F$ of $M$ with $ r_M(F) \leq r_M(E)- 2$, we have $|E|-|F|\geq t(r_M(E)-r_M(F))+1$.
\end{claim}
\begin{proof}
Suppose there exists a flat $F'$ of $M$ with $r_M(F')\leq r_M(E)-2$ and $|E|-|F'|= t(r_M(E)-r_M(F'))$. Then,
\begin{eqnarray*}
|c(E)|&=&\max_{F\colon r_M(F) \leq r_M(E)- 2} \{|F|+t(r_M(E)-r_M(F)-1)\}+1\\
&\geq &|F'|+t(r_M(E)-r_M(F')-1)+1\\
&=& |F'|+|E|-|F'|-t+1\\
&=&|E|-t+1.
\end{eqnarray*}
Let $F$ be a flat such that $r_M(F)=r_M(E)-1$. Since $M$ satisfies the (tPDB) condition, we have \begin{eqnarray*}
|E|-|F|&\geq& t(r_M(E)-r_M(F))\\
&=& t.
\end{eqnarray*}
It follows that $|c(E)| - 1 \geq |E|-t \geq |F|$ and as such $|c(E)|> |F|$. It follows that any flat of $M$ satisfies one of the inequalities in Corollary \ref{interiorcolors}. 

Next we will show that $M$ contains $t$ pairwise disjoint rainbow bases, which contradicts our assumption of $M$.
By Lemma \ref{colordisjointextension}, it suffices to construct $t$ pairwise disjoint rainbow independent sets $I_1, \ldots, I_t$ such that any flat $F$ of $M$ satisfies Inequality \eqref{extensioninequality} in Lemma \ref{colordisjointextension}. We say a color $c$ has \textit{multiplicity} $k$ in $M$ if the number of elements with color $c$ in $M$ is $k$.
Let $m_1\geq m_2\geq\ldots\geq m_s\geq 2$ be the multiplicities of all colors $c_1, c_2,\ldots, c_s$ respectively that have multiplicity at least two in $M$. We have a few cases:

\begin{description}

\item Case 1: $m_1\geq t$. In this case, place one element of color $c_1$ in each of the independent sets $I_1, \ldots, I_t$. We have
$$\sum_{i=1}^t |I_i|=t=(t-1)+|c\left(\cup_{i=1}^t I_i\right)|.$$
Hence by Lemma \ref{subindependentlargeenough} and then by Theorem \ref{colordisjointextension}, $M$ has $t$ pairwise disjoint rainbow bases, contradicting the choice of $M$.

\item Case 2:  $m_1\leq t-1$. In this case, we can place one element of color $c_1$ in each of $\{I_1,I_2,\ldots,I_{m_1}\}$. Then in a similar fashion, place one element of color $c_2$ in each of $I_{m_1+1}, \ldots, I_{m_1 + m_2}$ (with $I_{t+i}\equiv I_i$). Perform this greedy construction for all colors in $E_2$ (in the order of decreasing multiplicity) until all elements of $E_2$ have been added to $I_1, I_2, \ldots, I_t$ or until each $I_i$ receives two elements. 

We claim that there exists a swapping of elements in the $I_i$'s so that the resulting sets are all independent. Arrange the $I_i$ so that $I_1,I_2,\cdots I_k$ are 2-cycles and the remaining are independent. Consider first $I_1=\{x,y\}$. We claim there is an $I_j$, $j>1$ and a swapping of the elements so that $I_1$ and $I_j$ are independent and rainbow. Consider any set $I_j$, $1<j\leq k$. If there is no swapping of the elements in $I_1$ and $I_j$ to make both sets independent and rainbow, then (C2) above implies that all pairs of elements in $I_1$ and $I_j$ are 2-cycles. It follows that for all  $1<j\leq k$, $I_j$ contributes at least 1 unique element to $cl_2(x)$. 

Now consider any $I_j$, $k<j\leq t$. If $|I_j|=1$ and $I_j\cup \{y\}$ is independent, then we are done. Otherwise $I_j\cup \{y\}$, and thus by (C2), $I_j\cup\{x\}$, are 2-cycles. It follows that each such $I_j$ contributes a unique element to $cl_2(x)$. If $|I_j|=2$ and there is no swapping of elements in $I_1$ and $I_j$ which leaves both sets rainbow and independent, then $x$ combined with some element of $I_j$ must be a 2-cycle. It follows that for each $k<j\leq t$, each $I_j$ contributes at least 1 unique element to $cl_2(x)$. Therefore, the $I_2,\cdots, I_t$ contribute at least $t-1$ unique elements to $cl_2(x)$ and $\{x\}$ and $\{y\}$ both contribute one. It follows that $|cl_2(x)|\geq t+1$ which is a contradiction to Claim \ref{2cycles}. As such, we can find such a swapping. Make the swap and reorder the $I_i$'s so that the first $k'$ are 2-cylces and the remaining are independent and then repeat this argument with the new $I_1$ to find a new swapping. Clearly after each swap, the number of dependent sets decreases by at least one. It follows that we can continue in this way until all sets are independent. 

If all elements of $E_2$ have been added to $I_1, I_2, \ldots, I_t$, then we have that the left hand side of Inequality \ref{eqn:extension} is in fact equal to $|E\setminus T|$. Since $M$ satisfies the (tPDB) condition, Inequality \ref{eqn:extension} holds.
Hence by Lemma \ref{colordisjointextension}, $M$ contains $t$ pairwise disjoint bases, contradicting the choice of $M$.
Otherwise, each $I_i$ has exactly two elements. Since there are at most $t$ colors appearing in the $I_i$'s, we then have
$$\sum_{i=1}^t|I_i|\geq (t-1)+c\left(\cup_{i=1}^tI_i\right),$$
in which case by Lemma \ref{subindependentlargeenough} and then by Lemma \ref{colordisjointextension}, $M$ contains $t$ pairwise disjoint rainbow bases, contradicting the choice of $M$ again.
\end{description} 
\end{proof}
\begin{claim}\label{edgeremoval}
There exists an element $e\in E_2$ whose removal maintains that in the restriction matroid $M\setminus e$, each flat $F$ satisfies the (tPDB) condition. 
\end{claim}
 By Claim \ref{BaseCase}, every flat $F$ of $M$ with $ r_M(F) \leq r_M(E)- 2$ has $|E\setminus F|\geq t(r_M(E)-r_M(F))+1$. It follows that the removal of any element maintains that every flat with $  r_M(F) \leq r_M(E)- 2$ satisfies the (tPDB) condition. As such, the only way $M\setminus e$ violates the (tPDB) condition is if there exists a flat of $M$ with $r_M(F)=r_M(E)-1$ and $|E\setminus F|=t(r_M(E)-r_M(F))$. We claim there can only be one such flat. 

Suppose there are two such ``bad" flats $F_1$ and $F_2$. It is well known that the intersection of two flats is a flat. 
By Inequality \eqref{rank_submodular}, we have
\begin{align*}
    r_M(F_1\cap F_2)&\leq r_M(F_1)+r_M(F_2) -r_M(F_1\cup F_2)\\
    &= (r_M(E)-1) + (r_M(E)-1) -r_M(E)\\
    &=r_M(E)-2.
\end{align*}
Furthermore, 
\begin{eqnarray*}
|E|-|F_1\cap F_2|&\leq &|E|-|F_1|+|E|-|F_2|\\
&=&t(r_M(E)-r_M(F_1))+t(r_M(E)-r_M(F_2))\\
&=&t\cdot 2\\
&\leq & t(r_M(E)-r_M(F_1\cap F_2)).
\end{eqnarray*}
Since $M$ satisfies the (tPDB) condition, and $F_1\cap F_2$ is a flat satisfying $r_M(F_1\cap F_2)\leq r_M(E)-2$, the previous inequality contradicts Claim \ref{BaseCase}. As such, there can only be one ``bad" flat.

Let $F_0$ be the ``bad" flat.
Now since any flat $F$ with $r_M(F)\leq r_M(E)-2$ satisfies $|E|-|F|\geq t(r_M(E)-r_M(F))+1$, we have
\begin{eqnarray*}
|c(E)|&=&\max_{F\colon r_M(F) \leq r_M(E)- 2} \{|F|+t(r_M(E)-r_M(F)-1)\}+1\\
&\leq & |E|-t.
\end{eqnarray*}
Let $C_1$ be the colors on $E$ which are used only once. Since $|c(E)|<|E|$, at least one color must belong to $E_2$. As such, $|C_1|\leq |c(E)|-1\leq |E|-t-1$. It follows that 
\begin{eqnarray*}
|E_2|&=&|E|-|C_1|\\
&\geq & |E|-(|E|-t-1)\\
&=&t+1.
\end{eqnarray*}
Since $|E|-|F_0|=t(r_M(E)-r_M(F_0))=t$, and $|E_2|\geq t+1$, there exists an element $e\in E_2\cap F_0$. This element satisfies the claim. As such, the claim is proven. 
\end{proof}
Let $e\in E_2$ be the element granted by Claim \ref{edgeremoval} and remove $e$ from $M$. Let $M\setminus e=(E\setminus e,\I')$ be the resulting matroid. Notice that since $M$ has $t$ pairwise disjoint bases, deleting an element can only destroy one of them. It follows that $r_{M\setminus e}(E\setminus e)=r_M(E)$. Now by the minimality of $M$,
\begin{eqnarray*}
ar(M\setminus e,t)&\leq& \max_{F\colon r_{M\setminus e}(F) \leq r_{M\setminus e}(E)- 2} \{|F|+t(r_{M\setminus e}(E)-r_{M\setminus e}(F)-1)\}\\
&\leq & \max_{F\colon r_{M}(F) \leq r_{M}(E)- 2} \{|F|+t(r_M(E)-r_M(F)-1)\}\\
&<& |c(E\setminus e)|.
\end{eqnarray*}
As such, $M\setminus e$ contains $t$ pairwise disjoint rainbow bases. Since $r_{M\setminus e}(E\setminus e)=r_{M}(E)$, it follows that these $t$ pairwise disjoint rainbow bases are indeed bases for $M$. This contradicts the choice of $M$ and completes the proof of Theorem \ref{main}. 

\section{Applications to Various Matroids}\label{applications}

In this section, we derive Theorem 2 in \cite{Lu-Meier-Wang-AntiRamsey} using Theorem \ref{main} and apply Theorem \ref{main} to some matroids. 

\subsection{Graphical Matroid}\label{Graphic}

Let $G=(V(G),E(G))$ be a multigraph and let  $M=(E(G),\I)$ be the matroid whose independent sets are the forests of $G$. A basis for $M$ is then a spanning tree of $G$. Let $P$ be a partition of the vertex set of $G$. Adding any edge of $G$ into a partition must decrease the number of parts and thus increase the rank of the non-crossing edges. This justifies that the flats of $M$ correspond to a partition of the vertex set of $G$ and that the edges within the parts correspond exactly to the elements in $F$. The rank $r_M(F)$ of the flat is the sum of the ranks of the parts of $P$ so that 
\begin{eqnarray*}
r_M(F)&=&|V(G)|-|P|\\
&=& r_M(E)-|P|+1.
\end{eqnarray*}
The previous line implies $|P|-2=r_M(E)-r_M(F)-1$. Theorem \ref{main} then reads exactly as

\begin{theorem}\label{mainST2}
For any multigraph $G$, if there is a partition $P_0$ of vertices of $G$ satisfying $|E(G)|-|E(P_0,G)|<t(|P_0|-1)$, then $ar(G,t)=|E(G)|$. Otherwise, 
\begin{equation}
ar(G,t)=\max_{P\colon |P| \geq 3} \{|E(P,G)|+t(|P|-2)\},
\end{equation}
where the maximum is taken among all partitions $P$
(with $|P|\geq 3$) of the vertex set of $G$.
\end{theorem}

\noindent
The above is exactly the main result of \cite{Lu-Meier-Wang-AntiRamsey}.

\subsection{Bi-circular Matroids}
Let $G=(V(G),E(G))$ be a connected multigraph with at least one cycle. One can define the \emph{bi-circular matroid} $M_{bc}(G)=(E(G),\I)$ as the matroid whose independent sets are \emph{pseudoforests}, that is a subgraph of $G$ in which every connected component has at most one cycle. We call a subgraph of $G$ \emph{unicyclic} if every connected component has exactly one cycle. We define the anti-Ramsey number of $t$ edge disjoint rainbow unicyclic graphs as the maximum number of colors $ar_{bc}(G,t)$ such that there exists a coloring of the edges of $G$ with $ar_{bc}(G,t)$ colors that contains no $t$ edge disjoint rainbow unicyclic subgraphs.

In the bi-circular matroid, bases are unicyclic subgraphs while flats are subgraphs of $G$ whose connected components are either an induced graph of $G$ or a tree. The rank of a subgraph $F$ is $|V(G)|- \tau(F)$ where $\tau(F)$ be the number of tree components in $F$. 
In particular, $r(M_{bc}(G))=|V(G)|$. Theorem \ref{main} can be rephrased under this context as follows.

\begin{theorem}\label{bicircular}
For any multigraph $G$ at least one cycle and at least two vertices, if there is a subgraph $F_0$ of $G$ satisfying $|E(G)|-|E(F_0)|<t \cdot \tau(F_0)$, then $ar_{bc}(G,t)=|E(G)|$. Otherwise, 
\begin{equation}
ar_{bc}(G,t)=\max_{F\colon \tau(F)\geq 2} \{|E(F)|+t(\tau(F)-1)\},
\end{equation}
where the maximum is taken among all subgraphs $F$
(with $\tau(F)\geq 2$) of $G$.
\end{theorem}

Now we consider the complete graph $K_n$. The flats of $K_n$ are subgraphs whose components are either a clique or a tree. 
To apply Theorem \ref{bicircular}, we need to maximize $|F|+t\cdot \tau(F)$. Among all $F$ with fixed $\tau(F)=x$, 
$|F|+t\cdot \tau(F)$ is maximized when $F=K_{n-x}$ plus $x$ isolated vertices.
Consider the function $f(x)=\binom{n-x}{2}+tx$. It is clear that the degenerate case of Theorem \ref{main} now reads that if $\binom{n}{2}< \max_{0\leq x \leq n}\{f(x)\}$, then $ar_{bc}(K_n,t)=\binom{n}{2}$. Since $f(x)$ is concave up, the maximum occurs either when $x=0$ or $x=n$. Notice that $f(0)=\binom{n}{2}$. The degenerate condition would then read $\binom{n}{2}<\binom{n}{2}$ which is never satisfied. Since $f(n)=tn$, after some simplification, the degenerate condition reads that if $n<2t+1$, then $ar_{bc}(K_n,t)=\binom{n}{2}$. 

Considering the non-degenerate case of Theorem \ref{bicircular} we have

\begin{eqnarray*}
ar_{bc}(K_n,t)&=&\max_{F\colon \tau(F)\geq 2} \{|F|+t(\tau(F)-1)\}\\
&=&\max_{2\leq x\leq n}\{f(x)-t\}\\
&=&\max\{f(2)-t,f(n)-t\}.
\end{eqnarray*}

The last line above holds by the concavity of $f(x)$. This gives the following theorem:
\begin{theorem}\label{rainbowpseudoforests}
For all positive integers $n\geq 3$ and $t\geq 1$. We have
$$ar_{bc}(K_n,t) = \begin{cases} 
\binom{n}{2} &   \textrm{ for } n\leq 2t,\\
t(n-1)   & \textrm{ for } 2t+1\leq n \leq 2t+3,\\
\binom{n-2}{2}+t   & \textrm{ for } n \geq 2t+4.
 \end{cases}$$
\end{theorem}

\subsection{Cographic Matroids}
Given a connected multigraph $G$, a \emph{cut} is a collection of edges $C$ so that deleting $C$ disconnects $G$. A \emph{bond} is a minimal cut of $G$. The anti-Ramsey number of $t$ edge disjoint rainbow bonds $ar_b(G,t)$ asks for the maximal number of colors one can put on the edges of $G$ with no $t$ edge disjoint rainbow bonds of size $|E(G)|-|V(G)|+1$. 

A \emph{cographic matroid}, denoted by $M^*(G)$, is a matroid with the ground set $E(G)$ such that the circuits are the bonds of $G$. All bases have the form $E\setminus T$, where $T$ is a spanning tree of $G$.
In other words, the cographic matroids are the dual of the graphic matroids. In particular, when $G$ is a planar graph, the cographic matroid $M^*(G)$ is the same as the graphical matroid $M(G^*)$. Here $G^*$ is the dual graph of the planar graph $G$.  

The rank of $M^*(G)$ is $|E(G)|-|V(G)|+1$. The rank of any subset $S\subseteq M^*$ is 
\begin{eqnarray*}
r_{M^*}(S)&=&r_M(E\setminus S)+|S|-r_M(E)\\
&=&|V(G)|-q(E\setminus S)+|S|-|V(G)|+1\\
&=&|S|-q(E\setminus S)+1,
\end{eqnarray*}
where $q(S)$ is the number of components in $S$. A subset $F\subseteq E$ is a flat of $M^*$ if and only if $E\setminus F$ is a union of cycles. Theorem \ref{main} then reads that if there is a flat $F$ with $|E|-|F|<(r_{M^*}(E)-r_{M^*}(F))$, then $ar_b(G,t)=|E|$ and otherwise,
$$ar_b(G,t)=\max_{F\subseteq E \colon r_{M^*}(F)\leq r_{M^*}(E)-2}\{|F|-t(r_{M^*}(E)-r_{M^*}(F)-1)\}.$$

\subsection{Signed Graphs}

A \emph{signed graph} $ G=(V( G),E( G))$ is a multigraph were each edge is assigned a positive or negative sign. A subgraph of $ G$ is called positive, or negative, if the product of the signs on its edges is positive, or negative respectively. Suppose $G$ is connected and contains at least one negative cycle. A subset of edges is called \emph{balanced} if each cycle is positive. We call a subgraph of $G$ in which each component has exactly one negative cycle a \emph{signed unicyclic subgraph}. We define the anti-Ramsey number of $t$ edge disjoint rainbow signed unicyclic subgraphs, denoted $ar_{\pm}( G,t)$, as the maximum number of colors one can put on the edges of $ G$ such that $ G$ contains no $t$ edge-disjoint rainbow signed unicyclic subgraphs. 

One can define the signed graphical matroid $M=(E,\I)$ where $E$ are the edges of $ G$ and a set of edges are independent if and only if each component contains no cycles or exactly one cycle which is negative. For a subgraph $H$ of $G$, let $b( H)$ be the number of balanced components of $ H$. The rank of a subset $S\subseteq E$ satsifies $r_M(S)=|V(G)|-b(S)$. In particular, $r_M(E)=|V(G)|.$

Theorem \ref{main} then reads

\begin{theorem}\label{generalsigned}
Let $ G = (V( G),E( G))$ be a connected signed graph with at least 2 vertices, at least one negative cycle, and $t\geq 1$ a positive integer. If there exists a subgraph $F_0$ with $|E(G)|-|E(F_0)|<t\cdot b( F_0)$, then $ar_{\pm}(G,t)=|E(G)|$. Otherwise,
$$ar_{\pm}(G,t)=\max_{F \colon b(F)\geq 2} \{|E(F)|+t(b(F)-1)\}$$
where the maximum is taken over all subgraphs $F$ with $b(F)\geq 2$.
\end{theorem}

We consider the special case where $ G = \pm K_n$ the complete signed graph on $n$ vertices. I.e., $\pm K_n$ is the graph where each possible +1 or -1 edge exists between every pair of vertices and there are no loops.

We turn to maximizing $|E(F)|+t\cdot b(F)$. It is known that a signed graph is balanced if and only if it partitions into a pair of subgraphs, one of which is allowed to be empty, all edges within the parts are positive, and have negative edges connecting the parts \cite{Cartwright-Harary1979}. It follows that if $F$ is a flat of rank $n-b(F)$, each of the $b(F)$ balanced components partition into two complete positive subgraphs and all negative edges exist between. It follows that if $B$ is a balanced component of $F$, $|E(B)|=\binom{|B|}{2}$. The remaining parts are induced subgraphs of $\pm K_n$ which can be combined into one large induced subgraph of $\pm K_n$. Let $\{B_1,B_2,\cdots, B_{b(F)}\}$ be the balanced components of $F$. It follows that 
\begin{eqnarray*}
|E(F)|&=& \sum_{i=1}^{b(F)} \binom{|B_i|}{2}+2\cdot \binom{n-\sum_{i=1}^{b(F)} |B_i|}{2}.
\end{eqnarray*}
The previous is maximized when all the $B_1,B_2,\cdots, B_{b(F)}$ are a single point. We define 
$$f(s)= 2\cdot \binom{n-s}{2}+t s.$$
By the concavity of $f(s)$, the maximum is attained at either $s=0$ or $s=n$. When $s=0$, $f(0)=2\cdot \binom{n}{2}$ so that the degenerate case reads $2\cdot \binom{n}{2}<2\cdot \binom{n}{2}$ which never happens. If $s=n$, we have the condition that $2\binom{n}{2}<tn$ or $n<t+1$. Considering the non-degenerate case of Theorem \ref{main}, we have
\begin{eqnarray*}
ar_{\pm}(\pm K_n,t)&=&\max_{F : r_M(F)\leq r_M(E)-2}\{|F|+t(r_M(E)-r_M(F)-1)\} \\
&=&\max_{2\leq s \leq n}\{f(s)-t\}\\
&=&\max\{f(2)-t,f(n)-t\}.
\end{eqnarray*}

We have the following:

\begin{theorem}\label{rainbowsigned}
For all positive integers $n\geq 3$ and $t\geq 1$. We have
$$ar_{\pm }(\pm K_n,t) = \begin{cases} 
2\cdot \binom{n}{2} &   \textrm{ for } n\leq t,\\
t(n-1)   & \textrm{ for } t+1\leq n \leq t+3,\\
2\cdot \binom{n-2}{2}+t   & \textrm{ for } n \geq t+4.
 \end{cases}$$
\end{theorem}

\subsection{Transversal Matroids}
Transversal matroids were first studied by Edmonds and Fulkson \cite{Edmonds-Fulkerson1965} in 1965. Given a set family $\mathcal{A}=\{A_j\colon  j\in J$\}, let $E=\cup_{j\in J}A_j$ be the ground set.
A \emph{partial transversal} of $\mathcal{A}$ is a subset $T$ of $E$ such that
there is an injective map $\phi\colon T\to J$ with
$t\in A_{\phi(t)}$ for all $t\in T$. When $|T|=|J|$, partial transversals are called \emph{transversals}.

The transversal matroid, denoted by
$M[\mathcal{A}]$ is a matroid consisting of the ground set $E$ and all partial transversals as independent sets. Assume transversals 
of $\mathcal{A}$ exist, then bases are just transversals. 

From the set family $\mathcal{A}=\{A_j\colon  j\in J\}$, one can construct a bipartite incidence graph $G$ with bipartite vertex sets $J\cup E$, where $(j,x)$ is an edge of $G$ if and only if $x\in A_j$. The partial transverals are simply the second ends of matchings in $G$. For any $S\subseteq E$, the rank $r_M(S)$ is simply the maximum size of matchings in the induced subgraph $G[J\cup S]$.
Let $ar_{tr}(\mathcal{A}, t)$ be the maximum number $m$ such that there exists a coloring of $E$ using $m$ colors so that no $t$ disjoint rainbow transversals exist. Then, Theorem \ref{main} reads that if there exists a flat $F$ with $|E|-|F|< t(r_M(E)-r_M(F))$ then $ar_{tr}(\mathcal{A},t)=|E|$; else $ar_{tr}(\mathcal{A},t)=\max_{F\colon r(F)\leq r(E)-2} \{|F|+t(r(E)-r(F)-1)\}.$





For any positive integer $n,k$, the \emph{uniform matroid} $U_{k,n}$ is a special transversal matroid with $E=[n]$ and $\mathcal{A}=(E,E,\cdots, E)$ repeated $k$ times. Equivalently, the bases of $U_{k,}$ are all subsets of size $k$. 
The flats of $U_{k,n}$ are any subset of size at most $k-1$ and $[n]$. Let $F$ be a flat with size $s$. Theorem \ref{main} yields that if $n-s<t(k-s)$, then $ar_{tr}(U_{k,n},t)=n$. Otherwise we have,
\begin{eqnarray*}
ar_{tr}(U_{k,n},t)&=&\max_{F\colon r_{U_{k,n}}(F) \leq r_{U_{k,n}}(E)- 2} \{|F|+t(r_{U_{k,n}}(E)-r_{U_{k,n}}(F)-1)\}\\
&=&s+ t\cdot k -t\cdot s -t \\
&=& t( k-1) +(1-t)s
\end{eqnarray*}
The maximum above is achieved when $s=0$. This gives the following theorem:

\begin{theorem}\label{UniformAR}
Let $E=[n]$ and fix $1\leq k\leq n$ a natural number. If $n<kt$, then $ar_{tr}(U_{k,n},t)=n$ Otherwise,
$$ar_{tr}(U_{k,n},t)=t(k-1).$$
\end{theorem}

\subsection{Linear Matroids}
Let $F$ be a field, $d\geq 2$, and $E$ a finite subset of $F^d$. A linear matroid $M=(E,\I)$ can be defined where $\I$ consists of all linearly independent subsets of $E$. We assume $E$ spans $F^d$. Bases are just bases of the vector space. Flats are just intersections of linear subspaces with $E$. The rank of a subset $S\subseteq E$ is just the dimension of the linear space spanned by $S$.

The anti-Ramsey number $ar_\ell(E,t)$ of $t$ disjoint rainbow bases is the maximum integer $m$ such that there exists a coloring of $E$ with $m$ colors that contains no $t$ disjoint rainbow bases. Theorem \ref{main} reads,

\begin{theorem}\label{inear}
Given a finite set $E$ of $F^d$ with $d\geq 2$, suppose $E$ spans $F^d$. If there is a subspace $S_0$ of $F^d$ satisfying $|E|-|E\cap S_0|<t \cdot (d-dim(S_0))$, then $ar_\ell(E,t)=|E|$. Otherwise, 
\begin{equation}
ar_\ell(E,t)=\max_{S\colon dim(S)\leq d-2} \{|S\cap E|+t(d-dim(S)-1)\},
\end{equation}
where the maximum is taken among all linear subspaces $S$ of dimension at most $d-2$.
\end{theorem}
A specific example can be given when $E$ is a cube containing the zero vector. Namely, 
 let $F$ be a field and $T$ a finite subset of $F$ with $|T|\geq 2$ and $0\in T$. We consider the set $E=T^d\subseteq F^d$.  Notice that for any subspace $S$ of dimension $s$, $|S\cap T^d|\leq |T|^s$. The equality is achieved when $S$ is defined by setting
 $d-s$ coordinates equal to $0$. Let $f(x)= |T|^x+t(d-x)$. Observe that $f(x)$ is concave upward. The maximum of $f(x)$ on any interval is reached at one of the ends. Since $f(d)=|T|^d$, the degenerate case is given by
 $|T|^d<f(0)=1+td$. When $|T|^d\geq 1+td$, we have $ar_\ell(T^d,t)=\max\{f(0),f(d-2)\}-t$.
 This implies the following result.

\begin{theorem}\label{FVSAR2}
For any finite subset $T$ of any field $F$ with $0\in T$ and $|T|\geq 2$, and for  any integer $d\geq 2$, we have
 $$ar_\ell(T^d,t) = \begin{cases} 
|T|^d &   \textrm{ if }  t>\frac{|T|^d-1}{d},\\
1+t(d-1) & \textrm{ if } \frac{|T|^{d-2}-1}{d-2} \leq t\leq \frac{|T|^d-1}{d},\\
|T|^{d-2}+t  & \textrm{ if } t\leq \frac{|T|^{d-2}-1}{d-2}.
 \end{cases}$$

\end{theorem}


\begin{thebibliography}{1}

\bibitem{Akbari-Alipour07}
S. Akbari, A. Alipour, Multicolored trees in complete graphs, \textit{J. Graph Theory}, {\bf 54}, (2007), 221-232.

\bibitem{Bialostocki-Voxman01}
 A. Bialostocki and W. Voxman, On the anti-Ramsey numbers for spanning trees. \textit{Bull. Inst. Combin. Appl.} {\bf 32} (2001), 23-26.

\bibitem{Broersma-Li97}
H. Broersma, X. Li, Spanning trees with many or few colors in edge-colored graphs, \textit{Discuss. Math. Graph Theory} \textbf{17(2)} (1997), 259-269.

\bibitem{Carraher-Hartke17}
J. M. Carraher and S. G. Hartke, Eulerian circuits with no monochromatic transitions in edge-colored digraphs
with all vertices of outdegree three, \textit{SIAM J. Discrete Math.} {\bf 31}(1) (2017), 190-209.


\bibitem{CHH16}
J. Carraher, S. Hartke, P. Horn, Edge-disjoint rainbow spanning trees in complete graphs, \textit{European J. Combin.}, {\bf 57} (2016), 71-84.

\bibitem{Cartwright-Harary1979}
D. Cartwright and F. Harary, Balance and clusterability: An overview. \textit{Perspectives in Social Network Research},  (1979), 25-50 

\bibitem{Edmonds65}
J. Edmonds, Lehman’s switching game and a theorem of Tutte and Nash-Williams.
J. Res. Nat. Bur. Standards Sect. B 69B (1965), 73–77.

\bibitem{Edmonds68}
J. Edmonds, Matroid Partition, \textit{Math. of the Decision Sciences}, Amer. Math Soc. Lectures in Appl. Math. \textbf{11} (1968), 335-345.

\bibitem{Edmonds70}
J. Edmonds, Submodular functions, matroids and
certain polyhedra. in \textit{Combinatorial structures and
their applications}, eds. R. Guy, H. Hanani, N.
Sauer and J. Schonheim, Pages 69-87, 1970.



\bibitem{Edmonds-Fulkerson1965}
J. Edmonds, D. R. Fulkerson, Transversals and Matroid Partition. \textit{Journal of Research of the National Bureau of Standards Section B Mathematics and Mathematical Physics} (1965): 147-153.

\bibitem{ESS75}
P. Erd\H{o}s, M. Simonovits, and V. S\'os, Anti-Ramsey theorems, in \textit{Infinite and Finite Sets} (Colloq. Keszthely 1973). Colloq Math Soc Janos Bolyai \textbf{10} (1975), 633-643.

\bibitem{FGLX2021}
 C. Fang, E. Gy\H{o}ri, M. Lu and J. Xiao, On the anti-Ramsey number of forests, \textit{Discrete Appl. Math.} \textbf{291} (2021), 129-142.

\bibitem{FMO10}
S. Fujita, C. Magnant, and K. Ozeki, Rainbow generalizations of Ramsey
theory: a survey. \textit{Graphs Combin} \textbf{26} (2010), 1-30.

\bibitem{Gorgol2016}
I. Gorgol, Anti-Ramsey numbers in complete split graphs. \textit{Discrete Math.} \textbf{339} (2016), 1944-1949.

\bibitem{GLS2020}
 R. Gu, J. Li and Y. Shi, Anti-Ramsey numbers of paths and cycles in hypergraphs,
\textit{SIAM J. Discrete Math.} \textbf{34(1)} (2020), 271-307.

\bibitem{Haas-Young12}
R. Haas and M. Young, The anti-Ramsey number of perfect matching, \textit{Discrete Math.}, {\bf 312} (2012), 933-937.

\bibitem{Jahanbekam-West16}
S. Jahanbekam, D.B. West, Anti-Ramsey problems for t edge-disjoint rainbow spanning subgraphs: cycles, matchings, or trees,  \textit{J Graph Theory} \textbf{82(1)} (2016), 75-89.

\bibitem{Jia-Lu-Zhang2021}
Y. Jia, M. Lu, Y. Zhang, Anti-Ramsey Problems in Complete Bipartite Graphs for t Edge-Disjoint Rainbow Spanning Trees, \textit{Graphs Combin.} \textbf{37} (2021), 409–433.

\bibitem{Jiang-Pikhurko2009}
T. Jiang and O. Pikhurko, Anti-Ramsey numbers of doubly edge-critical graphs, \textit{J. Graph Theory.} \textbf{61} (2009) 210-218.


\bibitem{Jiang-West04}
T. Jiang and D. B. West, Edge-colorings of complete graphs that avoid polychromatic
trees, \textit{Discrete Math.} \textbf{274} (2004), 137-145.

\bibitem{LSS2019}
Y. Lan, Y. Shi and Z. Song, Planar anti-Ramsey numbers of paths and cycles, \textit{Discrete Math.}, \textbf{342} (2019), 3216-3224.

\bibitem{Lu-Meier-Wang-AntiRamsey}
L. Lu, A. Meier, and Z. Wang, Anti-Ramsey number of edge-disjoint rainbow spanning trees in all graphs, \emph{arXiv:2104.12978}[math.CO].

\bibitem{Lu-Wang-AntiRamsey}
L. Lu and Z. Wang, Anti-Ramsey number of edge-disjoint rainbow spanning trees,
\textit{SIAM J. Discrete Math.} {\bf34(1)} (2020), 271-307.


\bibitem{Montella-Nemann05}
J. J. Montellano-Ballesteros and V. Neumann-Lara, An anti-Ramsey theorem
on cycles. \textit{Graphs Combin} \textbf{21} (2005), 343-354.

\bibitem{Nash-Williams67}
C. St. J. A. Nash-Williams, An application of matroids to graph theory, in: \textit{Theory of Graphs - International Symposium} (Rome, 1966; P. Rosenstiehl, ed.), Gordon and Breach, New York, and Dunod, Paris, 1967, pp. 263-265.

\bibitem{Nash-Williams61}
C. St. J. A. Nash-Williams, Edge disjoint spanning trees of finite graphs. \textit{J. London Math. Soc.}, {\bf 36} (1961), 445-450.

\bibitem{Oxley1992} James G. Oxley, \textit{Matroid Theory}, Oxford University Press, Oxford, 1992.

\bibitem{Schrijver03}
A. Schrijver, \textit{Combinatorial optimization. Polyhedra and efficiency. Vol. B}, Volume 24 of Algorithms and Combinatorics. Springer-Verlag, Berlin, 2003. Matroids, trees, stable sets, Chapters 39-69.

\bibitem{Suzuki06}
K. Suzuki, A necessary and sufficient condition for the existence of a heterochromatic spanning tree in a graph, \textit{Graphs Combin.} {\bf 22(2)} (2006), 261-269.

\bibitem{Tutte61}
W. T. Tutte, On the problem of decomposing a graph into $n$ connected factors, \textit{Journal London Math. Soc}, {\bf 142} (1961), 221-230.

\bibitem{Whitney1935}
Hassler Whitney, On the Abstract Properties of Linear Dependence, \textit{American Journal of Mathematics}
{\bf 57} (1935), 509-533.

\bibitem{Wilson1973}
Robin J. Wilson, An Introduction to Matroid Theory, \textit{American Mathematical Monthly} {\bf 80} (1973),
500-525.
 
\bibitem{Xie-Yuan2020}
 T. Xie and L. Yuan, On the anti-Ramsey numbers of linear forests, \textit{Discrete Math.}, \textbf{343} (2020), 112130.

\bibitem{Yuan2021+}
L. Yuan, The anti-Ramsey number for paths, \textit{arXiv:2102.00807}[math.CO].

\bibitem{Yuan-Zhang2021+}
 L. Yuan and X. Zhang, Anti-Ramsey numbers of graphs with some decomposition
family sequences, \textit{arXiv:1903.10319}[math.CO].


\end{thebibliography}
\end{document}